\documentclass{article}

\usepackage{url}
\usepackage{amsmath,amssymb}
\usepackage{amsthm}
\usepackage{graphicx}
\usepackage{hyperref}
\usepackage[T1]{fontenc}
\usepackage{listings}
\usepackage[all]{xy}
\usepackage{graphicx}
\graphicspath{ {./Images/} }

\newtheorem{definition}{Definition}

\newtheorem{proposition}{Proposition}

\usepackage{tikz}
\usetikzlibrary{calc,decorations.markings,matrix,cd,decorations.pathmorphing,backgrounds}
\newcommand{\Act}{\mathbf{Act}}

\newcommand{\Ptr}{\mathbf{Ptr}}

\begin{document}

\title{Unitary magma actions}
\author{Nelson Martins-Ferreira\footnote{Polytechnic of Leiria}}
\date{\today}

\maketitle

\begin{abstract}

We introduce a novel concept of action for unitary magmas, facilitating the classification of various split extensions within this algebraic structure. Our method expands upon the recent study of split extensions and semidirect products of unitary magmas conducted by Gran, Janelidze, and Sobral.
Building on their research, we explore split extensions in which the middle object does not necessarily maintain a bijective correspondence with the Cartesian product of its end objects.
Although this phenomenon is not observed in groups or any associative semiabelian variety of universal algebra, it shares similarities with instances found in monoids through weakly Schreier extensions and certain exotic non-associative algebras, such as semi-left-loops.
Our work seeks to contribute to the comprehension of split extensions in unitary magmas and may offer valuable insights for potential abstractions of categorical properties in more general contexts.

\end{abstract}

\section{Introduction}

Recall that a unitary magma is a system $B=(B,+,0)$ consisting of a set $B$, a neutral element $0\in B$ and a binary operation $+$ with $b+0=b=0+b$ for all $b\in B$. As is well known, the neutral element is unique provided it exists and will occasionally be treated as a property (only for convenience). A morphism of unitary magmas from $B=(B,+,0)$ to $B'=(B',+,0)$ is a map of the underlying sets $f\colon{B\to B'}$ such that $f(b_1+b_2)=f(b_1)+f(b_2)$, for all $b_1,b_2\in B$, and $f(0)=0$. 

In recent years, the study of split extensions outside the classical algebraic structures has gained significant attention. This work focuses on unitary magmas, presenting a novel approach to classify a broad range of split extensions, including those not previously accounted for in existing literature.

Instead of split extensions of unitary magmas with a retraction map, as considered in \cite{gran2019},  we will consider a comparable but more general notion which will be called a \emph{retraction point}. A retraction point consists of a split epimorphism of unitary magmas (i.e. a point)
\begin{equation*}
\xymatrix{A \ar@<.5ex>[r]^(.5){p}  & B \ar@<.5ex>[l]^(.5){s} },\quad ps=1_B,
\end{equation*}
equipped with an injective map $k$, from a given set $X$ to the underlying set of $A=(A,+,0)$ and a retraction map $q\colon{A\to X}$ such that $kq(a)+sp(a)=a$ for all $a\in A$ (see Definition~\ref{def:Sch point} for more details).

We will say that the system $(A,k,q,s,p)$ is a retraction point from the set $X$ to the unitary magma $B$. The order used in the tuple $(A,k,q,s,p)$ reflects the identity $1_A=kq+sp$ which is a fundamental identity in the notion of semibiproduct recently introduced in \cite{martinsferreira2022,
martinsferreira2023a}.

The motivation for considering a set $X$ instead of a unitary magma is rooted in the fact that, unlike the split extensions studied in \cite{gran2019}, which are classified by actions of the form $\xi : B \times X \to X$, retraction points will be classified by actions of the form $\varphi : X \times B \times X \times B \to X$, as detailed in Definition~\ref{def: action}. The retraction points that are split extensions in the sense of \cite{gran2019} constitute a special case where $\varphi(x, b, x', b') = x + \xi(b, x')$ for a given unitary magma structure $(X,+,0)$. 
However, in the general case, $X$ is a set that is tacitly assumed to be a unitary magma with a neutral element $q(0) \in X$ and a magma operation $x + x' = q(k(x) + k(x'))$ whenever a retraction point $(A,k,q,s,p)$ is given from $X$ to $B$.

Compared to the split extensions examined in \cite{gran2019}, where the underlying set is bijective to $X\times B$, retraction points introduce a notable distinction. The underlying set of a retraction point is bijective to a specialized subset of $X\times B$, in which pairs $(x,b)$ must meet the condition $q(k(x)+s(b))=x$. This distinct property separates retraction points from their (Schreier) split extension counterparts.

An important aspect pertained to morphisms of retraction points is clarified. Such aspect is not evident when the condition $q(k(x)+s(b))=x$ is imposed for all $x \in X$ and $b \in B$, as seen in \cite{gran2019}, which is a natural condition when abstracting from groups or semiabelian varieties but not from monoids, unitary magmas or unital categories. The study of weakly Schreier split extensions, especially in the category of monoids, has been motivated by removing this condition (see e.g. \cite{graham} and refrences there).
For a retraction point $(A, k, q, s, p)$ from the set $X$ to the unitary magma $B$, the pair $(k,s)$ is jointly strongly epimorphic (as established in Proposition \ref{prop: retraction point properties}). Consequently, any morphism compatible with $k$ and $s$ is automatically compatible with $p$. However, compatibility with $q$ may not be guaranteed and can potentially differ.
This yields two distinct relations between retraction points with fixed ends.

When a set $X$ and a unitary magma $B$ are fixed, the set of all retraction points from $X$ to $B$, denoted by $\Ptr(X,B)$, is equipped with two binary relations:
\begin{enumerate}
    \item  $(A,k,q,s,p)\leq (A',k',q',s',p')$  \emph{iff} there exists a morphism $\alpha\colon{A\to A'}$ such that $\alpha k=k'$ and $\alpha s=s'$. 
\item $(A,k,q,s,p)\sim (A',k',q',s',p')$  \emph{iff} there exists a morphism $\alpha\colon{A\to A'}$ such that $\alpha k=k'$, $\alpha s=s'$ and $q'\alpha=q$.   

\end{enumerate}

The relation $\leq$ is a preorder and will be analysed in a forthcoming paper together with some connections with topology (in the spirit of \cite{fibrous}) which are motivated by the list of classical examples displayed in Section \ref{sec: egs}.

The relation~$\sim$ is an equivalence relation (Proposition \ref{prop: ssfl for ptr}). The primary contribution of this paper (Proposition \ref{prop: main thm}) is the classification which states that
$$\Act(X,B)\cong\Ptr(X,B)\slash\sim$$
where $\Act(X,B)$ denotes the collection of all $B$-actions with underlying set $X$.

This work is organized along the following lines. Firstly, the definition of retraction point is provided, along with some of its properties. 
Secondly, the notion of action of unitary magmas is introduced together with a procedure for associating an action to every retraction point. Moreover, every action is shown to induce a semidirect product of unitary magmas which may be seen as a canonical representative for the class  of its equivalent retraction points. Finally, it is shown that several particular cases become closer to the familiar formulas for group or monoid actions together with some examples and counter-examples. In particular is is shown that the unit sphere can be seen as a retraction point from the unit circle to the closed interval $[-1,1]$ with a suitable unitary magma structure.

\section{Retraction points of unitary magmas}\label{sec: retraction points}

The name retraction point stems from a split extension (also called a point) equipped with a retraction map thus making it a semibiproduct  \cite{martinsferreira2022,
martinsferreira2023a}.

\begin{definition}\label{def:Sch point}
Let $X$ be a set and $B=(B,+,0)$ a unitary magma. A \emph{retraction point} from $X$ to $(B,+,0)$ is a tuple $(A,k,q,s,p)$, represented as a (semibiproduct) diagram
\begin{equation}\label{diag: semibiproduct 1}
\xymatrix{X\ar@{.>}@<-.5ex>[r]_(.5){k} & A \ar@<.5ex>[r]^(.5){p} \ar@{.>}@<-.5ex>[l]_(.5){q} & B \ar@<.5ex>[l]^(.5){s} },\quad kq+sp=1_A,
\end{equation}
where $q$ and $k$ are maps, $A=(A,+,0)$ is a unitary magma, $p$ and $s$ are morphisms of unitary magmas, and moreover, $ps=1_B$, $qk=1_X$, $p(k(x))=0\in B$ for all $x\in X$, and $q(s(b))=q(0)$ for all $b\in B$.
\end{definition}

Given a retraction point from $X$ to $(B,+,0)$ we will tacitly assume the (unique) unitary magma structure on $X$ for which the map $k$ is a morphism. Then $k$ is a kernel of $p$  and is equipped with a retraction map $q$. The condition $kq+sp=1_A$ is called the semibiproduct condition \cite{martinsferreira2022,
martinsferreira2023a}. Depending on context we will sometimes write $A$ for a unitary magma $A=(A,+,0)$ or for a retraction point $A=(A,k,q,s,p)$. We will often write $B$ in the place of $(B,+,0)$.

\begin{proposition}\label{prop: retraction point properties}
Let $X$ be a set, $B=(B,+,0)$ a unitary magma and consider any retraction point $(A,k,q,s,p)$ from $X$ to $B$:
\begin{enumerate}
\item for all $x,x'\in X$ and $b,b'\in B$:
\begin{enumerate}
    \item $kq(0)=0$ and $q(k(x)+kq(0))=x=q(kq(0)+k(x))$;
    \item $q(k(x)+s(b))=q((k(x)+0)+(0+s(b))=q((k(x)+s(b))+(0+0)=q((0+0)+(k(x)+s(b))$;
    \item $q(s(b)+s(b'))=0$;
    \item $q((k(x)+s(b))+(k(x')+s(b')))=q(k(w)+s(b+b'))$, with $w=q((k(u)+s(b))+(k(v)+s(b')))$, $u=q(k(x)+s(b))$ and $v=q(k(x')+s(b'))$;
\end{enumerate}
\item when $k$ is seen as a morphism then it is the kernel of $p$, the pair $(k,s)$ is jointly strongly epimorphic and $p$ is the cokernel of $k$;
\item for every unitary magma $Z$, every map $u\colon{X\to Z}$ and every morphism $v\colon{B\to Z}$:
\begin{enumerate}
\item there exists at most one morphism $w\colon{A\to Z}$ such that $wk=u$ and $ws=v$;
\item if there exists a morphism $w\colon{A\to Z}$, with $wk=u$ and $ws=v$, then $w(a)=uq(a)+vp(a)$, for all $a\in A$, and $uq(k(x)+k(x'))=u(x)+u(x')$, for all $x,x'\in X$;
\item there exists a morphism $w\colon{A\to Z}$ with $wk=u$ and $ws=v$ if and only if
\begin{equation}
uq(a+a')+vp(a+a')=(uq(a)+vp(a))+(uq(a')+vp(a'))
\end{equation}
for every $a,a'\in A$;
\end{enumerate}
\item for every unitary magma $Z$, every map $f\colon{Z\to X}$ and every morphism $g\colon{Z\to B}$:
\begin{enumerate}
\item there exists at most one morphism $h\colon{Z\to A}$ for which $qh=f$ and $ph=g$;
\item if there exists a morphism $h\colon{Z\to A}$, with $qh=f$ and $ph=g$, then $h(z)=kf(z)+sg(z)$, for all $z\in Z$;
\item there exists a morphism $h\colon{Z\to A}$ with $qh=f$ and $ph=g$ if and only if
\begin{equation}
kf(z+z')+sp(z+z')=(kf(z)+sg(z))+(kf(z')+sg(z'))
\end{equation}
for all $z,z'\in Z$;
\end{enumerate}
\item the pullback of any morphism $g\colon{Z\to B}$ along $p\colon{A\to B}$ induces a retraction point $(A\times_{B}Z,\langle k,0\rangle,q\pi_1,\langle sg,1\rangle,\pi_2)$ as displayed
\begin{equation}\label{diag: pullback stability}
\xymatrix{X\ar[d]_{1_X} \ar[r]^(.44){\langle k,0\rangle} & A\times_{B} Z \ar[d]_{\pi_1} \ar@<.5ex>[r]^{\pi_2} & Z\ar[d]_{g} \ar@<.5ex>[l]^(.44){\langle sg,1\rangle}
\\
X\ar[r]^{k} & A \ar@<.5ex>[r]^{p} & B \ar@<.5ex>[l]^{s}}
\end{equation}
\item a retraction point of the form $(B,k',q',s',p')$, say from $Y$ to $C=(C,+,0)$, is composable with $(A,k,q,s,p)$ if and only if
\begin{equation}\label{eq: stable under composition}
(k(x)+sk'(y))+ss'(c)=k(x)+s(k'(y)+s'(c))
\end{equation}
for all $x\in X$, $y\in Y$ and $c\in C$;
\item if $k(x)+s(b+b')=(k(x)+s(b))+s(b')$ for all $x\in X$ and $b,b'\in B$ then every retraction point of the form $(B,k',q',s',p')$ is composable with $(A,k,q,s,p)$;
\item the map $\langle q,p\rangle\colon{A\to X\times B}$ is a bijection if and only if $q(k(x)+s(b))=x$ for all $x\in X$ and $b\in B$.
\end{enumerate}
\end{proposition}
\begin{proof}
\begin{enumerate}
\item It is immediate to observe $kq(0)=kq(0)+sp(0)=0$ while the other conditions are easily checked. This will be used in Proposition \ref{prop: every retraction point is an action} to show that every retraction point induces an action via the assignment $(x,b,x',b')\mapsto q((k(x)+s(b))+(k(x')+s(b')))$. 
\item
The map $k$ becomes a morphism of unitary magmas as soon as $X$ assumes the tacit structure in which $q(0)$ is the neutral element whereas the operation is $x+x'=q(k(x)+k(x'))$. Indeed, $k(0)=k(q(0))=kq(0)+0=kq(0)+sp(0)=0\in A$ and since $pk=0$,  $k(x+x')=k(x)+k(x')$ is equivalent to $qk(x+x')=q(k(x)+k(x'))$ which holds if and only if $x+x'=q(k(x)+q(x'))$. The morphism $k$ is the kernel of $p$ since every morphism $f\colon{Z\to A}$ with $pf=0$ factors uniquely through $X$ via the map $qf$ which is necessarily a morphism of unitary magmas. Indeed, $kqf=kqf+0=kqf+spf=(kq+sp)f=f$ and  $qf(z+z')=qf(z)+qf(z')$. The fact that every $a\in A$ can be decomposed as $a=kq(a)+sp(a)$ ensures that $(k,s)$ is jointly strongly epimorphic and as a consequence, since we also have $pk=0$, the morphism $p$ is the cokernel of $k$. 

\item It is easily checked that $(A\times_{B}Z,\langle k,0\rangle,q\pi_1,\langle sg,1\rangle,\pi_2)$ is a retraction point. Indeed, for every  $(a,z)\in A\times Z$ with $p(a)=g(z)$ we have $(a,z)=(kq(a)+sp(a),z)=(kq(a)+sg(z),z)=(kq(a),0)+(sg(z),z)$. 

\item Recall that a retraction point $(B,k',q',s',p')$ from a set $Y$ to a unitary magma $C$ is said to be composable (see Proposition 4.5 in \cite{martinsferreira2023a}, see also \cite{gran2019}) with $(A,k,q,s,p)$ as soon as the tuple $(A,\pi_1,q'',ss',p'p)$, where $$q''(a)=(kq(a)+sk'q'p(a),q'p(a)),$$ is a retraction point from  $A\times_B Y$ to $C$ as illustrated.
    \begin{eqnarray}
        \vcenter{\xymatrix{ & A\times_B Y\ar@{->}@<-0.5ex>[d]_{\pi_1}\ar@<.5ex>[r]^(.6){\pi_2} & Y\ar@<-.5ex>@{->}[d]_{k'} \ar@{->}@<.5ex>@{->}[l]^(.35){\langle sk',1\rangle}\\
        X\ar@<-.5ex>[r]_{k} & A \ar@<-.5ex>@{..>}[l]_{q} \ar@<-.5ex>@{..>}[u]_{q''} \ar@<.5ex>[r]^{p} \ar@<.5ex>[d]^{p'p}  & B \ar@<.5ex>@{->}[l]^{s} \ar@{->}@<-.5ex>@{..>}[u]_{q'} \ar@<.5ex>[d]^{p'}\\
        & C \ar@<.5ex>@{->}[u]^{ss'} \ar@{=}[r] & C\ar@<.5ex>[u]^{s'}}}
    \end{eqnarray}
    The map $q''$ is well defined, $p(kq(a)+sk'q'p(a))=k'q'p(a)$ for all $a\in A$. Moreover, $p'pss'=1_C$, $p'p\pi_1=p'k'\pi_2=0$, $q''ss'=0$ and we observe
    \begin{eqnarray*}
        q''\pi_1 &=&\langle kq+sk'q'p,q'p \rangle \pi_1\\
        &=&\langle kq\pi_1+sk'q'p\pi_1,q'p\pi_1 \rangle \\
        &=&\langle kq\pi_1+sk'q'k'\pi_2,q'k'\pi_2 \rangle \\
        &=&\langle kq\pi_1+sp\pi_1,\pi_2 \rangle \\
        &=&\langle \pi_1,\pi_2 \rangle =1_{A\times_B Y}.
    \end{eqnarray*}
    It remains to analyse the condition $\pi_1q''+ss'p'p=1_A$. 
    
    Since $s=s1_B=s(k'q'+s'p')=sk'q'+ss'p'$  we have 
    \begin{align*}
     \pi_1q''+ss'p'p &=\pi_1q''+(sk'q'+ss'p')s'p'p\\
     &=\pi_1q''+(sk'q's'p'p+ss'p's'p'p)\\
     &=\pi_1q''+ss'p'p\\     
     &=(kq+sk'q'p)+ss'p'p
    \end{align*}
and by hypotheses   $(kq+sk'q'p)+ss'p'p=kq+(sk'q'p+ss'p'p)$ from which the desired result follows as
\begin{align*}
\pi_1q''+ss'p'p &= kq+(sk'q'p+ss'p'p)\\
&=kq+s(k'q'p+s'p'p)\\
&=kq+sp\\
&=1_A.    
\end{align*}
 Conversely, having $\pi_1q''+ss'p'p=1_A$, on the one hand we observe that $(kq+sk'q'p)+ss'p'p=1_A$ while on the other hand, since $s=sk'q'+ss'p'$, we observe $1_A=kq+sp=kq+(sk'q'p+ss'p'p)$. So, by choosing $x=q(a)$, $y=q'p(a)$ and $c=p''(a)$ we see that condition (\ref{eq: stable under composition}) holds.
\end{enumerate}
\end{proof}

The following result can be seen as the {Split Short Five Lemma for retraction points} of unitary magmas.

\begin{proposition}\label{prop: ssfl for ptr}
    Let $(A,k,q,s,p)$ and $(A',k',q',s',p')$ be two retraction points from $X$ to $B$. If there exists a morphism $\alpha\colon{A\to A'}$ such that $\alpha k=k'$, $\alpha s=s'$ and $q'\alpha =q$ then $\alpha$ is an isomorphism. 
\end{proposition}
\begin{proof}
    Since $\alpha k=k'$, $\alpha s=s'$ we also have that $p' \alpha=p$ since $(k,s)$ is jointly strongly epic. In order to show that $\alpha$ is an isomorphism it suffices to show that it has an inverse. The map $\beta(y)=kq'(y)+sp'(y)$, with $y\in A'$ is invertible and its inverse is $\alpha$. Indeed, $\alpha(\beta(y))=\alpha(kq'(y)+sp'(y))=k'q'(y)+s'p'(y)=y$ for all $y\in A'$, and $\beta(\alpha(a))=kq'\alpha(a)+sp'\alpha(a)=kq(a)+sp(a)=a$, for all $a\in A$.
\end{proof}

This means that the relation $A\sim A'$ if and only if there exists a morphism $\alpha\colon{A\to A'}$ such that $\alpha k=k'$, $\alpha s=s'$ and $q'\alpha =q$ is an equivalence relation. A notion of $B$-action will be introduced in the following section and it will be shown that there is a one-to-one correspondence between the set of all $B$-actions with underlying set $X$ and the set of equivalence classes $\Ptr(X,B)\slash\sim$ for the equivalence relation $\sim$.

If in the definition of $\sim$ we do not require the compatibility of $\alpha$ with $q$ and $q'$, that is, if condition $q'\alpha=q$ is not a requirement, then the relation is no longer an equivalence relation but merely a preorder. 

For further reference we formalize the notion of equivalent retraction points.

\begin{definition}\label{def: equiv for retraction points}
Let  $A=(A,k,q,s,p)$ and $A'=(A',k',q',s',p')$ be two retraction points from $X$ to $B$. We will say that $A$ is equivalent to $A'$, written as $A\sim A'$, if (and only if) there exists a morphism $\alpha\colon{A\to A'}$ such that $\alpha k=k'$, $\alpha s=s'$ and $q'\alpha =q$.
\end{definition}

\section{Unitary magma actions}

In this section, we consider a fixed unitary magma $B=(B,+,0)$ and introduce a map $\varphi: X \times B \times X \times B \to X$, where $X$ is a set containing an element $0_X$. For simplicity, we will write $\varphi(0_X,0,0_X,0)$  as $\varphi(0,0,0,0)$. We believe this notation is clear and there is no need to write $0_X$ for $0\in X$ nor $0_B$ for $0\in B$.

\begin{definition}\label{def: action} Let $B=(B,+,0)$ be a unitary magma. A \emph{$B$-action} is a pair $(X,\varphi)$ where $X$ is a set, $$\varphi\colon X\times B\times X\times B\to X$$ is a map, and:
\begin{enumerate}
\item there exists $0\in X$ such that for all $x\in X$,
\begin{equation}\label{eq:act1}
\varphi(x,0,0,0)=x=\varphi(0,0,x,0);
\end{equation}
\item for all $x\in X$ and $b\in B$
\begin{equation}\label{eq:act2}
\varphi(x,b,0,0)=\varphi(x,0,0,b)=\varphi(0,0,x,b);
\end{equation}
\item for  all $b,b'\in B$
\begin{equation}\label{eq:act3}
\varphi(0,b,0,b')=0;
\end{equation}
\item the map $\varphi_{00}(x,b)=\varphi(x,0,0,b)$ is such that for all $x,x'\in X$ and  $b,b'\in B$,  
\begin{equation}\label{eq:act4}
\varphi(x,b,x',b')=\varphi_{00}(\varphi(\varphi_{00}(x,b),
b,\varphi_{00}(x',b'),b'),b+b').
\end{equation}
\end{enumerate}
\end{definition}

We will often use $\varphi_{00}(x,b)$ as a shorthand for $\varphi(x,0,0,b)$.
The following proposition shows that the subset of $X\times B$, consisting of those pairs $(x,b)$ for which $\varphi_{00}(x,b)=x$, is closed under the magma operation
\begin{equation}\label{eq:magma op 1}
(x,b)+(x',b')=(\varphi(x,b,x',b'),b+b')
\end{equation}
which is defined for all pairs $(x,b),(x',b')\in X\times B$.

\begin{proposition}\label{prop: well defined semidirect prod} Let $(X,\varphi)$ be a $B$-action. The set 
\begin{equation}\label{eq: semidirect prod - the set}
X\rtimes_{\varphi} B=\{(x,b)\in X\times B\mid \varphi(x,0,0,b)=x\}
\end{equation}
is closed under the binary operation displayed in equation (\ref{eq:magma op 1}) and moreover $(X\rtimes_{\varphi} B,+,(0,0))$ is a unitary magma.
\end{proposition}
\begin{proof}
 If $\varphi(x,0,0,b)=x$ and $\varphi(x',0,0,b')=x'$ then $\varphi(\varphi(x,b,x',b'),0,0,b+b')=\varphi(x,b,x',b')$ follows from (\ref{eq:act4}). From (\ref{eq:act2}) we have that $(0,0)$ is the neutral element.
\end{proof}

This motivates the definition of semidirect product of unitary magmas.

\begin{definition}\label{def: semi-direct product}
Let $(X,\varphi)$ be a $B$-action. The semidirect product of $(X,\varphi)$ and $B$, denoted by $X\rtimes_{\varphi} B$, is the subset $\{(x,b)\mid \varphi(x,0,0,b)=x\}\subseteq X\times B$ equipped with the magma  operation displayed in equation $(\ref{eq:magma op 1})$ and the neutral element $(0,0)\in X\rtimes_{\varphi} B$.
\end{definition}

The unitary magma $(X\rtimes_{\varphi} B, +, (0,0))$ will be denoted simply as $X\rtimes_{\varphi} B$. Furthermore, the set $X$ is tacitly considered as a unitary magma with $0\in X$ as the neutral element for the magma operation $x+x'=\varphi(x,0,x',0)\in X$.

The category of all $B$-actions will be denoted by  $\Act(B)$ and its morphisms will be called $B$-morphisms.

\begin{definition}\label{def: B-morphism (alternative)}
Let $(X,\varphi)$ and $(X',\varphi')$ be two  $B$-actions. A map $f\colon{X\to X'}$ is said to be a \emph{$B$-morphism} from $(X,\varphi)$ to $(X',\varphi')$ if $f(0)=0$ and the following condition is satisfied for all $x,x'\in X$ and $b,b'\in B$:
\begin{equation}\label{eq: f is B-morphism}
\varphi'(f(x),b,f(x'),b')=f(\varphi'(x,b,x',b')).
\end{equation}
\end{definition}

\begin{proposition}\label{prop: B-morphism}
Let $(X,\varphi)$ and $(X',\varphi')$ be two  $B$-actions and  $f\colon{X\to X'}$ a $B$-morphism from $(X,\varphi)$ to $(X',\varphi')$. The assignment 
\begin{equation}\label{eq: (x,b)-->(g(x,b),b)}
(x,b)\mapsto (f(x),b)
\end{equation}
is a magma morphism from $X\rtimes_{\varphi} B$ to $X'\rtimes_{\varphi'} B$.
\end{proposition}
\begin{proof}
    The assignment determines a well defined map $g(x,b)=(f(x),b)$ from $X\rtimes_{\varphi} B$ to $X'\rtimes_{\varphi'} B$. Indeed, a particular instance of condition \ref{eq: f is B-morphism} gives $\varphi'(f(x),0,f(0),b)=f(\varphi'(x,0,0,b))$ from which we conclude that if $(x,b)\in X\rtimes_{\varphi} B$ then $(f(x),b)\in X\rtimes_{\varphi'} B$. Since $f(0)=0$ we have $g(0,0)=(0,0)$ and once again condition \ref{eq: f is B-morphism} ensures that $g((x,b)+(x,b'))=g(x,b)+g(x,b')$ for all $(x,b),(x',b')\in X\rtimes_{\varphi} B$.
\end{proof}

In the following section we detail some familiar properties of semidirect products.

\section{Semidirect products of unitary magmas}

Subsequently, a demonstration will be provided on how actions can be employed for the classification of the special class of split extensions of unitary magmas involving retraction points. This classification includes the cases examined in \cite{gran2019} as a particular instance. Prior to delving into this, a review of some familiar properties of semidirect products as split extensions will be conducted.
 
\begin{proposition}\label{prop: induced split ext} If $(X,\varphi)$ is a $B$-action then $(X,\varphi(\text{-},0,\text{-},0),0)$ is a unitary magma and
\begin{equation}\label{diag: associated split ext}
\xymatrix{X\ar[r]^(.34){\langle 1,0 \rangle} & X\rtimes_{\varphi} B \ar@<.5ex>[r]^(.64){\pi_{B}} & B \ar@<.5ex>[l]^(.34){\langle 0,1 \rangle} }
\end{equation}
is a split extension of unitary magmas. Moreover, $(x,b)\in X\rtimes_{\varphi}B$ if and only if $(x,b)=(x,0)+(0,b)$.
\end{proposition}
\begin{proof}
Condition (\ref{eq:act1}) ensures that $0\in X$ exists and is the neutral element for the magma operation $x+x'=\varphi(x,0,x',0)$ in $X$ which turns the assignment $x\mapsto (x,0)\in X\rtimes_{\varphi} B$ into a morphism. By construction $\pi_B$ is a morphism. The assignment $b\mapsto (0,b)\in X\rtimes_{\varphi} B$ is a morphism via condition  (\ref{eq:act3}). The magma operation on $X\rtimes_{\varphi} B$ gives $(x,0)+(0,b)=(\varphi(x,0,0,b),b)$ so that $(x,b)\in X\rtimes_{\varphi}B$ if and only if $(x,b)=(x,0)+(0,b)$.
\end{proof}

The following proposition details some useful properties of semidirect products of unitary magmas.

\begin{proposition}\label{prop: morphism in and out}
Let $(X,\varphi)$ be a $B$-action and consider its associated split extension as displayed in $(\ref{diag: associated split ext})$. Then:
\begin{enumerate}
\item for every two morphisms $u$ and $v$ of unitary magmas as illustrated
\begin{equation}
\xymatrix{X \ar[rd]_{u} \ar[r]^(.35){\langle 1,0 \rangle} & X\rtimes_{\varphi} B \ar@{-->}[d]_{w} \ar@<.5ex>[r]^(.65){\pi_{B}} & B \ar@<.5ex>[l]^(.44){\langle 0,1 \rangle} \ar[ld]^{v}
\\ & Z}
\end{equation}
there exists at most one morphism  $w\colon{X\rtimes_{\varphi} B\to Z}$ such that $w(x,0)=u(x)$ and $w(0,b)=v(b)$ and moreover if such morphism exists it is necessarily of the form $w(x,b)=u(x)+v(b)$. Furthermore, the map $w(x,b)=u(x)+v(b)$ is a magma homomorphism from $X\rtimes_{\varphi} B$ to $Z$ if and only if 
\begin{equation}\label{eq: w=u+v is morphism}
u(\varphi(x,b,x',b'))+v(b+b')=(u(x)+v(b))+(u(x')+v(b')).
\end{equation}
\item  given a map of (pointed) sets $f\colon{Z\to X}$ and a morphism of unitary magmas  $g\colon{Z\to B}$, as illustrated,
\begin{equation}
\xymatrix{& Z \ar[rd]^{g} \ar@{.>}[ld]_{f} \ar@{-->}[d]\\ 
X  \ar[r]^(.45){\langle 1,0 \rangle} & X\rtimes_{\varphi} B  \ar@<.5ex>[r]^(.55){\pi_{B}} & B \ar@<.5ex>[l]^(.34){\langle 0,1 \rangle} 
}
\end{equation}  
the assignment $z\mapsto (f(z),g(z))$ is a morphism of unitary magmas if and only if
\begin{equation}\label{eq: f crossed hom}
f(z_1+z_2)=\varphi(f(z_1),g(z_1),f(z_2),g(z_2)).
\end{equation}
\end{enumerate}
\end{proposition}

Equation $(\ref{eq: w=u+v is morphism})$ is analogous to the familiar notion of the image of $u$ commuting with the image of $v$. More 
in particular, when $\varphi(x,b,x',b')$ is of the form $x+\xi(b,x')$, for some map $\xi\colon{B\times X\to X}$, then equation (\ref{eq: f crossed hom}) becomes the familiar condition for $f$ being a $(\xi,g)$-crossed-homomorphism in the sense that  $f(z_1+z_2)=f(z_1)+\xi(g(z_1),f(z_2))$.

The following proposition is analogous to the split short five lemma.

\begin{proposition}\label{prop:ssfl}
Let $(X,\varphi)$ and $(X',\varphi')$ be two  $B$-actions and consider the following diagram
\begin{equation}\label{diqg: morphism of semidirect prods}
\xymatrix{X\ar[d]_{f} \ar[r]^(.34){\langle 1,0\rangle} & X\rtimes_{\varphi} B \ar@{-->}[d]_{g} \ar@<.5ex>[r]^{\pi_B} & B\ar[d]^{1_B} \ar@<.5ex>[l]^(.34){\langle 0, 1\rangle}
\\
X'\ar[r]^(.34){\langle 1,0\rangle} & X'\rtimes_{\varphi'} B \ar@<.5ex>[r]^{\pi_B} & B \ar@<.5ex>[l]^(.34){\langle 0, 1\rangle}}
\end{equation}
in which $f$ is a morphism of unitary magmas.
\begin{enumerate}
\item The  following conditions are equivalent:
\begin{enumerate}
\item there exists a (unique) morphism $g\colon{X\rtimes_{\varphi} B\to X'\rtimes_{\varphi'} B}$ such that $\pi_B g=\pi_B$, $g\langle 1,0 \rangle=\langle f,0 \rangle$ and  $g\langle 0, 1\rangle=\langle 0,1 \rangle$,

\item there exists a (unique) morphism $g\colon{X\rtimes_{\varphi} B\to X'\rtimes_{\varphi'} B}$ such that  $g\langle 1,0 \rangle=\langle f,0 \rangle$ and  $g\langle 0, 1\rangle=\langle 0,1 \rangle$,

\item there exists a (unique) map $g_1\colon{X\rtimes_{\varphi} B\to X'}$ such that $g_1(x,0)=f(x)$ and 
\begin{equation}
g_1(\varphi(x,b,x',b'),b+b')=
\varphi'(g_1(x,b),b,g_1(x',b'),b')),
\end{equation}
for all $x,x'\in X$ and $b,b'\in B$,

\item the morphism $f$ is such that for all $x,x'\in X$ and $b,b'\in B$,
\begin{equation}\label{eq: is B-morphism}
\varphi'_{00}(f(\varphi(x,b,x',b')),b+b')=\varphi'(
\varphi'_{00}(f(x),b),b,\varphi'_{00}(f(x'),b'),b')
\end{equation}
 where $\varphi'_{00}(y,b)$ is a shorthand for $\varphi'(y,0,0,b)$.
\end{enumerate}

\item  The  following conditions are equivalent:
 \begin{enumerate}
 \item there exists a (unique) morphism $g\colon{X\rtimes_{\varphi} B\to X'\rtimes_{\varphi'} B}$ such that $\pi_{X'} g=f \pi_X$, $g\langle 1,0 \rangle=\langle f,0 \rangle$ and  $g\langle 0, 1\rangle=\langle 0,1 \rangle$,
 \item $f$ is a $B$-morphism.
 \end{enumerate} 
\item If $f$ is an isomorphism and a $B$-morphism then there exists a unique isomorphism  $g\colon{X\rtimes_{\varphi} B\to X'\rtimes_{\varphi'} B}$ such that $\pi_{X'} g=f \pi_X$, $g\langle 1,0 \rangle=\langle f,0 \rangle$ and  $g\langle 0, 1\rangle=\langle 0,1 \rangle$.
\end{enumerate}
\end{proposition}

\section{A classification of retraction points by actions}\label{sec: classification}

Let $B=(B,+,0)$ be a unitary magma and let $X$ be a set. We will denote by $\Ptr(X,B)$ the set of all retraction points from $X$ to $B$, that is, all five-tuples of the form $(A,k,q,s,p)$ in which $A$ is a unitary magma, $k\colon{X\to A}$ and $q\colon{A\to X}$ are maps, $s\colon{B\to A}$ and $p\colon{A\to B}$ are morphisms of unitary magmas and moreover, the following conditions are satisfied:
\begin{align}
ps=& 1_B\\
pk=& 0\\
qk=& 1_X\\
qs=& 0\\
kq+sp=& 1_A.
\end{align}
Note that condition $pk=0$ is saying that $p(k(x))=0\in B$ for every $x\in X$, while condition $qs=0$ is to be understood as $q(s(b))=q(0)\in X$ for all $b\in B$ (Definition \ref{def:Sch point}). Once again, $X$ will tacitly be considered as a unitary magma with neutral element $q(0)$ and operation $x+x'=q(k(x)+k(x'))$. We are thus considering a subclass of split epimorphisms of unitary magmas over $B$ whose kernel is bijective to the given set $X$. Moreover, the equation $k(y)+sp(a)=a$ can be solved for every $a\in A$ thus providing the retraction map $q\colon{A\to X}$ as $q(a)=y$.

As we will see there is a correspondence between retraction points and actions.

\begin{proposition}\label{prop: every retraction point is an action}
Every retraction point $(A,k,q,s,p)$ from $X$ to $B=(B,+,0)$ induces a $B$-action $(X,\varphi)$ with $\varphi=q((k+s)+(k+s))$.    
\end{proposition}
\begin{proof}
    The first item in Proposition \ref{prop: retraction point properties} clearly shows that the map $$\varphi\colon{X\times B\times X\times B\to X}$$ defined by $\varphi(x,b,x',b')=q((k(x)+s(b))+(k(x')+s(b')))$ is a $B$-action.  
\end{proof}

\begin{proposition}\label{prop: every action is a retraction point} Every $B$-action $(X,\varphi)$ induces a canonical retraction point 
\begin{equation}\label{diag: canonical retraction point}
\xymatrix{X\ar@{.>}@<-.5ex>[r]_(.34){\langle 1,0 \rangle} & X\rtimes_{\varphi} B  \ar@{.>}@<-.5ex>[l]_(.64){\pi_X} \ar@<.5ex>[r]^(.64){\pi_{B}} & B \ar@<.5ex>[l]^(.34){\langle 0,1 \rangle} .}
\end{equation}
\end{proposition}
\begin{proof}
    It is clear that the semibiproduct condition, namely $(x,b)=(x,0)+(0,b)$, is equivalent to $\varphi(x,0,0,b)=x$. See also Proposition \ref{prop: induced split ext}.
\end{proof}

The previous propositions establish a correspondence between the set of all $B$-actions with underlying set $X$, $\Act(X,B)$, and the set of all retraction points from $X$ to $B$, $\Ptr(X,B)$.

Let us denote by $F_{X,B}\colon{\Act(X,B)\to \Ptr(X,B)}$ the assignment $F_{X,B}(\varphi)=(X\rtimes_{\varphi}B,\langle 1,0\rangle,\pi_X,\langle 0,1\rangle,\pi_B)$ and by $\Phi_{X,B}\colon{\Ptr(X,B)\to \Act(X,B)}$ the assignment $\Phi(A,k,q,s,p)=q((k+s)+(k+s))$.

Recall the notion of equivalence $\sim$ from Definition \ref{def: equiv for retraction points}.

\begin{proposition}\label{prop: main thm}
    For every set $X$ and every unitary magma $B=(B,+,0)$,
    \begin{equation}
        \Act(X,B)\cong\Ptr(X,B)\slash\sim
    \end{equation}
\end{proposition}
\begin{proof}
    It suffices to show that the map
    $$\Phi\colon{\Ptr(X,B)\to \Act(X,B)}$$ is such that 
    $$A\sim A' \Leftrightarrow \Phi(A)=\Phi(A')$$ and we observe:
\begin{enumerate}
    \item if $A\sim A'$ then there exists $\alpha\colon{A\to A'}$ such that $q'\alpha=q$, $\alpha k=k'$ and $\alpha s=s'$ from which we conclude that $\Phi(A)=q((k+s)+(k+s))=q'\alpha((k+s)+(k+s))=q'((k'+s')+(k'+s'))=\Phi(A')$;
    \item if $\Phi(A)=\Phi(A')$ then in particular we have $q(k+s)=q'(k'+s')$ and put $\alpha(a)=k'q(a)+s'p(a)$ for all $a\in A$. It follows that $\alpha(k(x))=k'(x)+s'(0)=k'(x)$, $\alpha(s(b))=k'(0)+s'(b)=s'(x)$, and
    $q'\alpha(a)=q'(k'q(a)+s'p(a))=q(kq(a)+sp(a)=q(a)$ since $q(k+s)=q'(k'+s')$. It is now a straightforward verification to check that the hypotheses $\Phi(A)=\Phi(A')$ is all we need to prove $\alpha$ to be a morphism.
\end{enumerate}

    Moreover, the map $F$ is a canonical choice of representatives. Indeed, for every set $X$ and every unitary magma $B=(B,+,0)$, we have $\Phi_{X,B}(F_{X,B}(\varphi))=\varphi$ and $F_{X,B}(\Phi(A,k,q,s,p))\sim (A,k,q,s,p)$.
\end{proof}

\section{Examples and particular cases}\label{sec: egs}

It is clear that every diagram of unitary magmas of the form
\begin{equation}\label{diag: semibiproduct 1}
\xymatrix{X\ar@{->}@<-.0ex>[r]^(.5){k} & A \ar@<.5ex>[r]^(.5){p} 
 & B \ar@<.5ex>[l]^(.5){s} }
\end{equation}
with $ps=1_B$ and $pk=0$ becomes a retraction point as soon as there exists a map $q\colon{A\to X}$ with $qk=1_X$, $qs=0$ and $kq+sp=1_A$. 
Indeed, if for every $a\in A$ there exists $y\in X$ such that $k(y)+sp(a)=a$ then every assignment $q(a)=y$ with $q(k(x))=x$ and $q(s(b))=0$ becomes a retraction point $(A,k,q,s,p)$. 

It is worth noting that the restriction on $s\colon{B\to A}$ being a morphism of unitary magmas is not strictly necessary. This work adopts this restriction for two primary reasons: (a) to maintain consistency with the traditional treatment of split extensions, and (b) to streamline calculations by avoiding the introduction of factor-sets. However, further research is warranted to explore the implications of allowing $s$ to be a general map rather than a morphism, as indicated by the following classical examples:
\begin{align}
    \xymatrix{\mathbb{Z}\ar@{->}@<-.5ex>[r]_(.5){} & \mathbb{R} \ar@<.5ex>[r]^(.5){\text{mod}} \ar@{.>}@<-.5ex>[l]_(.5){\text{floor}} & [0,1[ \ar@{.>}@<.5ex>[l]^(.5){} }
    \\
        \xymatrix{{0}\ar@{->}@<-.5ex>[r]_(.5){} & \mathbb{R} \ar@<.5ex>[r]^(.5){\text{exp}} \ar@{->}@<-.5ex>[l]_(.5){\text{}} & \mathbb{R}^{+} \ar@{->}@<.5ex>[l]^(.5){\text{log}} }
        \\
        \xymatrix{\{-1,1\}\ar@{->}@<-.5ex>[r]_(.5){} & \mathbb{R}\setminus \{0\} \ar@<.5ex>[r]^(.5){x^2} \ar@{.>}@<-.5ex>[l]_(.5){\text{sign}} & \mathbb{R}^{+} \ar@<.5ex>[l]^(.5){\sqrt{x}} }
        \\
        \xymatrix{\mathbb{Z}\ar@{->}@<-.5ex>[r]_(.5){2\pi i} & \mathbb{C} \ar@<.5ex>[r]^(.5){\text{exp}} \ar@{.>}@<-.5ex>[l]_(.5){\text{}} & \mathbb{C}^{*} \ar@{.>}@<.5ex>[l]^(.5){\text{log}} }
        \\
        \xymatrix{\mathbb{R} \ar@{->}@<-.5ex>[r]_(.5){} & \mathbb{C} \ar@<.5ex>[r]^(.5){\text{imag}} \ar@{->}@<-.5ex>[l]_(.5){\text{real}} & \mathbb{R} \ar@{->}@<.5ex>[l]^(.5){\text{}} }
        \\
        \xymatrix{S^{1} \ar@{->}@<-.5ex>[r]_(.5){} & \mathbb{C}^{*} \ar@<.5ex>[r]^(.35){\text{abs}} \ar@{->}@<-.5ex>[l]_(.5){\text{arg}} & ]0,+\infty[ \ar@{->}@<.5ex>[l]^(.5){\text{}} } 
        \\
        \xymatrix{S^{1} \ar@{->}@<-.5ex>[r]_(.5){k} & S^2 \ar@<.5ex>[r]^(.5){p} \ar@{.>}@<-.5ex>[l]_(.5){q} & [-1,1] \ar@{->}@<.5ex>[l]^(.5){s} } 
\end{align}

All examples are standard except the last one, which serves as our motivating example. It is well-known that the unit sphere $S^2$ is not a Lie group. Nevertheless, we will show that it can be equipped with a unitary magma structure, allowing it to become a retraction point from $S^1$ to the closed interval $[-1,1]$, which is itself equipped with an appropriate unitary magma structure.

In order to establish the desired unitary magma structure, we first transform the closed interval $[-1,1]$ bijectively into $[0,+\infty]$. Concurrently, we consider the unit sphere $S^2$ as the Riemann sphere $\mathbb{C}\cup\{0\}$, as illustrated.
\begin{equation}\label{diag: eg1}
\xymatrix{S^1 \ar@{=}[d]_{} \ar[r]^(.44){\text{}} & \mathbb{C}\cup  \{\infty\}\ar[d]_{\cong} \ar@<.5ex>[r]^{\text{abs}} & [0,\infty] \ar[d]_{\cong} \ar@<.5ex>[l]^(.44){}
\\
S^{1}\ar[r]^{k} & S^2 \ar@<.5ex>[r]^{p} & [-1,1] \ar@<.5ex>[l]^{s}}
\end{equation}
Then, we will see that 
\begin{equation}\label{diag: eg2}
\xymatrix{S^{1} \ar@{->}@<-.5ex>[r]_(.5){} & \mathbb{C}\cup\{\infty\} \ar@<.5ex>[r]^(.5){\text{abs}} \ar@{.>}@<-.5ex>[l]_(.5){\text{}} & [0,+\infty] \ar@{->}@<.5ex>[l]^(.5){\text{}} } \end{equation}
which showcases the relationships between $S^1$, $\mathbb{C}\cup\{\infty\}$, and $[0,+\infty]$, turns out to be a retraction point. This result has significant implications for the structure of $S^2$ as it becomes a retraction point from $S^1$ to $[-1,1]$.

The primary obstacle hindering the wider acceptance of the otherwise natural multiplication table on the set $\{0,1,\infty\}$, namely,
\begin{equation}
\begin{tabular}{c|ccc}
$\cdot$ & 1 & 0 & $\infty$ \\\hline
1 & 1 & 0 & $\infty$ \\
0 & 0 & 0 & 1 \\
$\infty$ & $\infty$ & 1 & $\infty$ \\
\end{tabular}
\end{equation}
seems to be its non-associative nature. 
However, this should not diminish its value as an exceptional and instructive example of a unitary magma.

The previous table can even be extended to
\begin{equation}
\begin{tabular}{c|cccc}
$\cdot$ & 1 & $x'$ & 0 & $\infty$ \\\hline
1 & 1 & $x'$ & 0 & $\infty$ \\
$x$ & $x$ & $xx'$ & 0 & $\infty$ \\
0 & 0 & 0 & 0 & 1 \\
$\infty$ & $\infty$ & $\infty$ & 1 & $\infty$ \\
\end{tabular}
\end{equation}
where $x$ and $x'$  are any two elements in a unitary magma $(X,\cdot,1)$ thus becoming a multiplication table that turns the set $X\sqcup \{0,\infty\}$ into  a unitary magma with neutral element $1\in X$. This procedure will first be applied to $X=(]0,+\infty[,\cdot, 1)$ and then to $X=(\mathbb{C}^{*},\cdot,1)$ in order to get $[0,+\infty]$ and $\mathbb{C}\cup\{\infty\}$ as unitary magmas. 

The maps $f\colon{]-1,1[\to]0,+\infty[}$ and $g\colon{]0,+\infty[\to ]-1,1[}$, with $f(x)=\frac{1+x}{1-x}$ and $g(y)=\frac{y-1}{y+1}$, induce a bijection between $[-1,1]$ and $[0,+\infty]$. As a result, we have a unitary magma $B=([-1,+1],\oplus,0)$ with operation $\oplus$ given by
\begin{equation}
\begin{tabular}{c|cccc}
$\oplus$ & 0 & $b'$ & $-1$ & $+1$ \\\hline
0 & 0 & $b'$ & $-1$ & $+1$ \\
$b$ & $b$ & $b\oplus b'$ & $-1$ & $+1$ \\
$-1$ & $-1$ & $-1$ & $-1$ & 0 \\
$+1$ & $+1$ & $+1$ & 0 & $+1$
\end{tabular}
\end{equation}
where, for every $b,b'\in]-1,+1[$,
\begin{equation}
b\oplus b'=\frac{\frac{1+b}{1-b}\cdot\frac{1+b'}{1-b'}-1}{\frac{1+b}{1-b}\cdot\frac{1+b'}{1-b'}+1}=\frac{b+b'}{bb'+1}.
\end{equation}

Finally, for $B=([-1,+1],\oplus,0)$, we consider the $B$-action $(X,\varphi)$ with $X=S^{1}$ the unit circle parameterized as $X=\{e^{ it}\in \mathbb{C}\mid -\pi< t\leq \pi\}$ and put $\varphi(e^{ it},b,e^{it' },b')=e^{i(t+t')}$ whenever $b,b'\in]-1,1[$ whereas in the cases $\pm 1\in B$ we choose to set $\varphi(x,\pm 1,x',\pm 1)=e^{0i}$. Then, $S^2$ can be identified with $X\rtimes_{\varphi} B$ since it consists of all pairs $(x,b)$ with $x\in X$ and $b\in ]-1,1[$ together with the special pairs $(e^{0i},-1)$ and $(e^{0i},+1)$ which correspond to the two poles on the sphere, as illustrated.  
\begin{equation}\label{diag: eg3}
\xymatrix{X \ar@{=}[d]_{} \ar[r]^(.44){\langle 1,0 \rangle} & X\rtimes_{\varphi} B  \ar[d]_{\cong} \ar@<.5ex>[r]^{\pi_B} & B \ar@{=}[d]_{} \ar@<.5ex>[l]^(.44){\langle 0,1\rangle}
\\
S^{1}\ar@<-.5ex>[r]_{k} & S^2 \ar@{.>}@<-.5ex>[l]_{q} \ar@<.5ex>[r]^{p} & [-1,1] \ar@<.5ex>[l]^{s}}
\end{equation}
In order to recover the topological sphere, if we take $S^2$ as the set of triples $(x,y,z)\in \mathbb{R}^3$ with $x^2+y^2+z^2=1$ then $p(x,y,z)=z$, $q(x,y,z)=e^{it}$ with $t\in ]-\pi,\pi]$ such that $x=\cos(t)$ and $y=\sin(t)$ if $z\neq\pm1$, while $q(x,y,z)=e^{0i}$ for $z=\pm 1$; moreover, $s(z)=(\sqrt{1-z^2},0,z)$.

A different kind of example is obtained by generalizing the natural order on the natural numbers, namely $m\leq n$ if and only if there exists $k$ such that $m+k=n$. Here, instead of commutativity and associativity, we will require the mediality law, also know as the midle interchange law, $$(x+y)+(z+w)=(x+z)+(y+w).$$ Let $(B,+,0)$ be a unitary magma satisfying the medial law, then the set $A=\{(x,b)\in B\times B\mid  x\geq b\}$ is a submagma of the cartesian product magma $B\times B$ where the relation $\geq$ is defined by $x\geq b$ if and only if there exists $ u\in B$ such that $x=u+b$. In this situation, $A$ becomes a retraction point from $B$ to itself with $q(x,b)=u$ for any choice of $u\in B$ such that $x=u+b$  as soon as $q(b,b)=0$ and $q(x,0)=x$; moreover,  $k(x)=(x,0)$, $p(x,b)=b$ and $s(b)=(b,b)$, as illustrated. 
\begin{equation}\label{diag: eg3}
\xymatrix{B \ar@{=}[d]_{} \ar[r]^(.44){\langle 1,0 \rangle} & B\rtimes_{\varphi} B  \ar@<-.5ex>[d]_{f} \ar@<.5ex>[r]^{\pi_2} & B \ar@{=}[d]_{} \ar@<.5ex>[l]^(.44){\langle 0,1\rangle}
\\
B\ar@<-.5ex>[r]_{k} & A \ar@<-.5ex>[u]_{g} \ar@{.>}@<-.5ex>[l]_{q} \ar@<.5ex>[r]^{p} & B \ar@<.5ex>[l]^{s}}
\end{equation}
Furthermore, if writing $q(x,b)$ as $(x-b)\in B$ whenever $x\geq b$ then the corresponding action is $\varphi(u,b,u',b')=((u+u')+(b+b'))-(b+b')$ and it is clear that the pairs $(u,b)\in B\rtimes_{\varphi} B$  are precisely those pairs $(u,b)$ for which $(u+b)-b=u$. The isomorphism between $A$ and $B\rtimes_{\varphi} B$ is observed as $g(x,b)=(x-b,b)$ and $f(u,b)=(u+b,b)$.

Upon identifying some traces of associativity, certain simplifications in the formula of $\varphi$ become apparent. When we say $k+(s+k)=(k+s)+k$ we really mean to say that $k(x)+(s(b)+k(x'))=(k(x)+s(b))+k(x')$ for all $x,x'\in X$ and $b\in B$, as well as for the similar situations in the following result.  

\begin{proposition}\label{prop: particular cases}
    Let $(A,k,q,s,p)$ be a retraction point from $X$ to $(B,+,0)$ and consider $\varphi=q((k+s)+(k+s))$ its classifying action. For
    \begin{align}
        x+_{b'}x':=\varphi(x,0,x',b')\\
        \xi^{x}(b,x'):=\varphi(x,b,x',0)\\
        \xi_{b'}(b,x'):=\varphi(0,b,x',b')\\
        \rho^{b}_{b'}(x):=\varphi(x,b,0,b')\\
        x+x':=\varphi(x,0,x',0)\\
        \xi(b,x'):=\varphi(0,b,x',0)\\
        \rho_{b'}(x):=\varphi(x,0,0,b')
    \end{align}
    \begin{enumerate}
        \item \emph{(kks)} if $k+(k+s)=(k+k)+s$ then
        \begin{equation}
            x+_{b'}x'=\rho_{b'}(x+x')
        \end{equation}
        \item (sks) if $s+(k+s)=(s+k)+s$ then
        \begin{equation}
            \xi_{b'}(b,x')=\rho^{b}_{b'}(\xi(b,x'))
        \end{equation}
        \item (1ks) if $1_A+(k+s)=(1_A+k)+s$ then
        \begin{equation}
            \varphi(x,b,x',b')=\rho^{b}_{b'}(\xi^{x}(b,x'))
        \end{equation}
        \item (kss) if $k+(s+s)=(k+s)+s$ then
        \begin{equation}
            \rho^{b}_{b'}(x)=\rho_{b+b'}(x)
        \end{equation}
        \item (ksk) if $k+(s+k)=(k+s)+k$ then
        \begin{equation}
            \xi^{x}(b,x')=x+_{b}\xi(b,x')
        \end{equation}
        \item (ks1) if $k+(s+1_A)=(k+s)+1_A$ then
        \begin{equation}
            \varphi(x,b,x',b')=x+_{b+b'}\xi_{b'}(b,x')
        \end{equation}
        \item (kks,ksk) if $k+(k+s)=(k+k)+s$ and $k+(s+k)=(k+s)+k$  then
        \begin{equation}
            \xi^{x}(b,x')=\rho_{b}(x+\xi(b,x'))
        \end{equation}
        \item (sks,kss) if $s+(k+s)=(s+k)+s$ and $k+(s+s)=(k+s)+s$  then
        \begin{equation}
            \xi_{b'}(b,x')=\rho_{b+b'}(\xi(b,x'))
        \end{equation}
        \item (1ks,kss,ksk) if $1_A+(k+s)=(1_A+k)+s$, $k+(s+s)=(k+s)+s$ and $k+(s+k)=(k+s)+k$  then
        \begin{equation}
            \varphi(x,b,x',b')=\rho_{b+b'}(x+\xi(b,x'))
        \end{equation}
                \item (ks1,kks,sks) if $k+(s+1_A)=(k+s)+1_A$, $k+(k+s)=(k+k)+s$ and $s+(k+s)=(s+k)+s$  then
        \begin{equation}
            \varphi(x,b,x',b')=\rho_{b+b'}(x+\xi(b,x'))
        \end{equation}
    \end{enumerate}
\end{proposition}
\begin{proof}
We will prove the last item, the others being similar. First, we observe that every $a+a'\in A$ can be written as
\begin{align*}
    a+a' &=(kq(a)+sp(a))+(kq(a')+sp(a'))\qquad  \text{retraction point}\\
    &=kq(a)+\left((sp(a)+kq(a'))+sp(a')\right)\qquad \text{ks1, sks}\\
    &=kq(a)+\left((kq(sp(a)+kq(a')) +sp(a))+sp(a')\right)\qquad \text{retraction point}\\
    &=k(q(a)+q(sp(a)+kq(a'))) +s(p(a)+p(a'))\qquad \text{ks1, kks.}
\end{align*}
Then, for $a=k(x)+s(b)$ and $a'=k(x')+s(b')$ we have $\varphi(x,b,x',b')=q(a+a')$ and as a consequence $\varphi(x,b,x',b')=\varphi(\varphi(x,0,\varphi(0,b,x',0),0),0,0,b+b')=\rho_{b+b'}(x+\xi(b,x'))$.
\end{proof}

The following special cases are noteworthy. By a left loop we mean a unitary magma $(B,+,0)$ with the property that for every $x,b\in B$ there exists a unique $u\in B$ such that $x=u+b$. We will also write $u$ as $x-b$ and observe that $b-b=0$, $b-0=b$, $(x-b)+b=x$ and $(u+b)-b=u$. 

\begin{proposition}\label{prop: particular cases}
    Let $(A,k,q,s,p)$ be a retraction point from $X$ to $(B,+,0)$ and consider $\varphi=q((k+s)+(k+s))$ its classifying action. For
    \begin{align}
        x+x':=\varphi(x,0,x',0)\\
        \xi(b,x'):=\varphi(0,b,x',0)\\
        \rho_{b'}(x):=\varphi(x,0,0,b')
    \end{align}
    \begin{enumerate}
        \item if $A=(A,+,0)$ is a monoid then:        
        \begin{enumerate} 
        \item $(X,+,q(0))$ and $(B,+,0)$ are monoids 
        \item $A$ is bijective to $$\{(x,b)\in X\times B\mid \rho_{b}(x)=x\}$$ via the assignments  $a\mapsto (q(a),p(a))$ and $(x,b)\mapsto k(s)+s(b)$
            \item  $\varphi(x,b,x',b')=\rho_{b+b'}(x+\xi(b,x'))$
            \item $(A,k,q,s,p)$ is composable with every retraction point of the form $(B,k',q',s',p')$, with $B$ the fixed unitary magma under consideration
        \end{enumerate}
        
        \item if $A=(A,+,0)$ is a left loop then
        \begin{enumerate}
            \item $kq(a)=a-sp(a)$ for all $a\in A$
            \item 
            $k\varphi(x,b,x',b')=((k(x)+s(b))+(k(x')+s(b')))-s(b+b')$
            \item $\rho_{b}(x)=x$, for all $x\in X$ and $b\in B$
            \item $k\xi(b,x')=(s(b)+k(x))-s(b)$, for all $x\in X$ and $b\in B$
            \item $A$ is bijective to $X\times B$ via the assignments  $a\mapsto (q(a),p(a))$ and $(x,b)\mapsto k(s)+s(b)$.
        \end{enumerate}
        \end{enumerate}
\end{proposition}

In a left loop $(A,+,0)$, the condition $(u+b)-b=u$ holds due to the uniqueness of solutions for equations of the form $u+b=x$, where $x,b\in A$. However, if we consider an additional operation $x-b$ on $A$ satisfying $b-b=0$ and $(x-b)+b=x$ but not necessarily $(u+b)-b=u$, we encounter what can be designated as a semi-left-loop, a left-loop in which the condition $(u+b)-b=u$ is not satisfied.
Even though we would be dealing with a semi-abelian algebra (\cite{bourn2002}), $A$ would not be bijective to the Cartesian product $X\times B$. Instead, there would exist a bijection between $A$ and $\{(x,b)\in X\times B\mid (k(x)+s(b)))-s(b)=k(x)\}$. For further details on this topic, we refer to \cite{gray}.

\paragraph{Acknowlegements:} this work has previously been funded by FCT/MCTES (PIDDAC) through the projects: Associate Laboratory ARISE LA-P-0112-2020; UIDP-04044-2020; UI\-DB-04044-2020; PAMI–ROTEIRO-0328-2013 (022-158); MATIS (CENTRO-01-0145-FEDER-000014 - 3362); CENTRO-01-0247-FED\-ER-(069665, 039969); as well as POCI-01-0247-FE\-DER-(069603, 039958, 03\-9863, 024533); by CDRSP and ESTG from the Polytechnic of Leiria.
Furthermore, the author acknowledges the project Fruit.PV and Fundação para a Ciência e a Tecnologia (FCT) for providing financial support through the CDRSP Base Funding project (DOI: 10.54499/UIDB/04044/2020).

\end{document}